\documentclass[11pt]{amsproc}
\usepackage{amstext,amsmath,amssymb,amsthm}
\usepackage{latexsym}
\usepackage{exscale}
\usepackage{color}
\usepackage{mdframed}

\newcommand{\R}{{ \mathbb  R  }}
\newcommand{\C}{  \mathbb  C }
\newcommand{\Z}{  \mathbb Z }
\newcommand{\N}{  \mathbb N }
\newcommand{\B}{  \mathcal B }
\newcommand{\ov}{\bar}
\renewcommand{\l}{\left\langle}
\renewcommand{\r}{\right\rangle}

\renewcommand{\Re}{\mbox{Re\,}}

\newtheorem{Thm}{Theorem}[section]
\newtheorem{Lemma}[Thm]{Lemma}
\newtheorem{Cor}[Thm]{Corollary}
\newtheorem{Prop}[Thm]{Proposition}
 \newcommand{\dsize}{\displaystyle}
 \newcommand{\cal}{\mathcal}
\numberwithin{equation}{section}

\newtheorem{Rem}{Remark}[section]

\begin{document}
\title[Riesz bases by replacement]{{\bf Riesz bases from orthonormal bases by replacement }}

\author {Laura De Carli}
\address{Department of Mathematics and statistics, Florida international University, Miami, FL 33199}
\email{decarlil@fiu.edu}

\author {Julian Edward}
\address{Department of Mathematics and statistics, Florida international University, Miami, FL 33199}
\email{edwardj@fiu.edu}

%

\begin{abstract}{
Given an orthonormal basis
  $  {\cal V}=
  \{v_j\} _{j\in\N}$ in a separable Hilbert space $H$
and a set of unit vectors
 $\B=\{w_j\}_{j\in\N}$,  we consider the sets $\B_N$ obtained by replacing the vectors $v_1, ...,\, v_N$ with vectors $w_1,\, ...,\, w_N$.   We show necessary and sufficient conditions that ensure that the sets $\B_N$ are Riesz bases of $H$ and we estimate the frame constants of the $\B_N$.   Then, we prove  conditions that ensure that $\B$ is a Riesz basis.    Applications to the construction of exponential bases on  domains of $\R^d$ are  also presented.}

\end{abstract}

\maketitle

 \bigskip
 \section {Introduction}
 Let $H$  be a separable Hilbert space  with inner product $\l\, , \,\r$ and norm $||v ||=\sqrt{\l v,\, v\r}$. Let
  $  {\cal V}=
  \{v_j\} _{j\in\N}$ be an orthonormal  basis on $H$.
  Let
 $\B=\{w_j\}_{j\in\N}$ be a set of unit vectors in $H$.
For every $N\ge 1$, we let
\begin{equation}\label{def-BN}\B_N=\{w_j\}_{1\leq j\leq N} \cup \{v_j\}_{j\ge N+ 1}.\end{equation}
That is, we replace the first $N$ vectors of  the  orthonormal  basis ${\cal V}$ with  the first $N$ vectors  of $\B$. We ask
 the following questions.

  \begin{enumerate}
 \item [a)]     Which assumptions on the $w_j$ can be made so that the $\B_N$'s   are Riesz bases of $H$?

 \item [b)] Can   the
 frame  bounds  of $\B_N$   be  explicitly computed?

 \item [c)] If the   $\B_N$ are    Riesz bases of $H$   with frame constants $A_N \leq B_N$ that satisfy  $0<\alpha \leq A_N\leq B_N\leq \beta<\infty$,  with $\alpha,\beta $ independent of $N$,
 is  $\B $  a Riesz basis of $H$?
\item[d)] If $\B$  and   the $\B_N$ are   Riesz bases of $H$, how are the frame constants of these bases related?

\end{enumerate}

We have recalled the definition of frame and Riesz basis  and other relevant preliminary   in   Appendix 1. In view of  \eqref{e2-frame} and \eqref{e2- Riesz-sequence},   the sets $\B_N$   are Riesz bases  of $H$ if there exists constants
$A_N ,\ B_N >0$ for which
\begin{enumerate}
\item   $ \dsize A_N  \leq \sum_{j=1 }^N|\l f, w_j\r|^2 + \sum_{j=N+1}^\infty|\l f, v_j\r|^2\leq B_N
$ for every $f\in H$  such that $||f||  =1$,  and
\item
$
\dsize A_N \leq
\Big\Vert \sum_{j=1 }^N a_j w_j +\sum_{j=N+1}^m b_jv_j\Big\Vert^2 \leq B_N
$ for every finite set of constants $\{a_1$, ... $a_N$, $b_{N+1}$, ... $b_m\}\subset\C$  for which $\sum_{j=1 }^N |a_j|^2 + \sum_{j=N} ^m|b_j|^2=1$.
\end{enumerate}

 \medskip
 The  inequalities on the  right-hand side of (1)   and (2) are easy to prove. Consider for example the second inequality in  (1). Since the $v_j$'s form an orthonormal basis,  we have that
 $
 \sum_{j=1 }^\infty|\l f, v_j\r|^2=||f||^2.
 $
By Cauchy-Schwartz inequality, $|\l f, w_j\r|^2\leq ||f||^2$ and so we have $$
\sum_{j=1 }^N|\l f, w_j\r|^2 + \sum_{j=N}^\infty|\l f, v_j\r|^2 \leq ||f||^2(1+ N).
$$
The proof of the right inequality in (2) is similar.
The   left inequalities in (1) and (2) are  not valid for every family of  $w_j$'s  and are much more difficult to prove.

In the following we will denote with $\B_0$ the set ${\cal V}=\{v_j\}_{j\in\N}$ and with  $\B_\infty$ the set $\B=\{w_j\}_{j\in\N}$. We will denote the optimal frame constants  of $\B_N$   with $A_N$ and $B_N$.  Clearly, $A_0=B_0=1$.

 \medskip
 To state our   main results we need to introduce some notation: for every $N\ge 1$, we let  $H_N'= \mbox{span}\{v_1, ...,\, v_N\}$ and $H_N''=  \mbox{span}\{v_{j}  \}_{j\ge N+1}$; note that $H''_N$ is the orthogonal complement of $H'_N$ and so  $H=H_N'\oplus H''_N$.

For every $u\in H$, we  denote with  $p_N'(u)$ and  $p_N''(u)$   the orthogonal projection of $u $ over $H'_N$  and $H''_N$. When there is no ambiguity, we will simply write $u'$ and $u''$ instead of $p_N'(u)$  and $p_N''(u)$.
The following theorem answers  the first of our questions.

 \begin{Thm}\label{T-2main}
 The following are equivalent:
 \begin{enumerate}
 \item[a)]
 $\B_N$ is a frame of $H$.
 \item [b)] The set $ \{p'_N(w_1),\, ...,\, p'_N(w_N) \}$ span $H'_N$.
  \item [c)] The set  $\{p'_N(w_1),\, ...,\, p'_N(w_N\}$ is linearly independent.
  \item[d)] $\B_N$ is a Riesz sequence  of $H$.
  \item[e)] $\B_N$ is a Riesz basis of $H$.
  \end{enumerate}

     \end{Thm}

Theorem \ref{T-2main}   shows that $\B_N$  is a  Riesz basis of $H$ if and only if $W'_N$ is a Riesz  basis of $H'_N  $. In infinite dimensional Hilbert spaces, it is not common to  have conditions that  ensure that a set   is a  Riesz sequence (or a frame)  if and only if  it is a Riesz basis.  See \cite{Dc}, \cite{DMT}.

\medskip
  We denote by $W'_N$ the set $\{p'_N(w_1),\, ...,\, p'_N(w_N) \}$ and with   $\tilde B_N$ the set  $ \{p'_N(w_1),\, ...,\, p'_N(w_N),\, v_{N+1}, \, v_{N+2}, ...\}$. We denote by  $ \dsize  U_N'$   the  $N\times N$ matrix whose elements are $ u_{i,j}=\l p'_N( w _i),\, p'_N (w _j)\r$ for $1\leq i,j\leq N$. Similarly, we define $ \dsize  U_N''$   the  infinite matrix whose elements are $ u_{i,j}=\l p''_N( w _i),\, p''_N (w _j)\r$ for $1\leq i,j\leq N$.
In view of Theorem \ref{T-2main}, $\tilde B_N$ is a Riesz basis of $H$ if and only if $W'_N$ is a basis of $H'_N$ and   if and only if $U'_N$ is nonsingular.
  We denote by $\tilde A_N$ and $\tilde B_N$ the frame constants of $\tilde \B_N$.

   If $  \vec v=(v_1, ...,\, v_n)\in \C^n$, we let $| \vec v\,|=\sqrt{|v_1|^2+... +|v_n|^2}$ be the Euclidean norm in $\C^n$. 
If there is no risk of confusion,  we will  also denote the norm of $\vec y\in \ell^2$ with $|\vec y\, |$.
   We prove the following

  \begin{Thm}\label{T-3main}

Assume that $W_N'$ is a basis of  $H_N'$.
Then,
 $  \B_N$ is a Riesz basis of $H$ with frame constants
$$
   B_N=\max_{\vec c\in\C^N:  |\vec c\,|\leq 1 }   \Big\{ \l  U_N'\, \vec c,\     \vec c \,    \r +\left(
   \sqrt{1- |\vec c\,|}+\sqrt{ \l  U_N''\, \vec c,\     \vec c\,\r }\right)^2\Big\},$$$$
A_N=\min_{\vec c\in\C^N:   |\vec c\,|\leq 1}   \Big\{ \l  U_N' \, \vec c,\     \vec c \,    \r +\left(
   \sqrt{1- |\vec c\,|}-\sqrt{ \l  U_N''\, \vec c,\     \vec c\,\r }\right)^2\Big\}.
 $$
 Furthermore, letting   $ \Lambda_N$, resp. $ \lambda_N$, be the maximum, resp.   minimum, eigenvalue of the matrix $ U_N'$, we have
\begin{equation}\label{e-main-eigen}
   B_N\ge\!  \tilde B_N\!=\!\max\{\Lambda_N,\, 1\};\     \lambda_N=\tilde A_N \!  \leq \max\{A_N, \tilde A_N\}\!\leq \min_{1\leq j\leq N} ||p_N'(w_j)||^2.
\end{equation}
\end{Thm}
 We note in passing that the matrix entries of $U_N'$ can be computed using the formula 
$$ u_{i,j}=\l p'_N( w _i),\, p'_N (w _j)\r =\sum_{k=1}^N\l w_i,v_k\r \overline{\l w_j,v_k\r}.$$

Theorem \ref{T-3main} answers question b).  Our next theorem answers questions c) and d).

\begin{Thm}\label{T-first-main} Assume that the $\B_N$'s are Riesz bases  of $H$ for every $N<\infty$.

a) If   $   \liminf_{N\to\infty }A_N> 0$ and $ \dsize \limsup_{N\to\infty }B_N< \infty $,  then $ B_\infty$ is a Riesz basis of $H$ with frame constants $A_\infty$, $B_\infty$ that satisfy
\begin{equation}\label{e-ineq-const} A_\infty\ge \liminf_{N\to\infty }A_N, \quad B_\infty\leq \limsup_{N\to\infty }B_N.\end{equation}

b)  if $\B_\infty$ is a Riesz basis of $H$ with frame constants $A_\infty$, $B_\infty$, then  the inequality \eqref{e-ineq-const} holds.
 \end{Thm}
We will also present an example proving that the first inequality in \eqref{e-ineq-const} can be strict.

We prove Theorems \ref{T-2main}, \ref{T-3main} and \ref{T-first-main} in Sections 2, 3 and 4. Preliminary results and definitions are contained in the  Appendix.
{In Section 5 we use  our results to prove sufficient conditions for the existence of exponential bases  in  $ L^2(D)$, where   $D\subset \R^d$ has finite Lebesgue measure $|D|$.    An {\it exponential  basis} of $L^2(D)$  is a Riesz basis in the form of  $E(\Lambda)=\{ e^{2\pi i \langle   \vec\lambda ,\,  x\rangle}\}_{ \vec\lambda\in\Lambda}$, where $\Lambda$ is a discrete set of $\R^d $.
Exponential bases  are important
to provide  unique and stable  representation of functions in $L^2(D)$  in terms of the functions in  $E(\Lambda)$,
 with  coefficients that are easy to calculate.
Unfortunately,  our understanding of exponential  bases is still very incomplete.
  There are very few examples of domains  in which it is known how to construct   exponential bases,
 and no example of domain  for which
 exponential bases are known  not to exist.  In particular, it is known that the disk in $\R^d$ does not have  orthogonal exponential bases (i.e. it is not {\it spectral}) but it is not known whether it  has   exponential Riesz bases or not.  The connection between tiling and spectral properties of   domains of $\R^d$  is deep and fascinating and has spur intense investigation since when B. Fuglede formulated his famous tiling $\iff$ spectral conjecture in \cite{F}. See  \cite{K} \cite{GL}, \cite{K2}  and  the   references
cited there.
 We plan to   to pursue further investigation on this problem in a subsequent paper.

  }

  \section{ Proof of Theorem \ref{T-2main}}

  Recall that    $H'_N =\mbox{span}\{v_1, ...,\, v_N\}$   and $H''_N$  is the orthogonal complement of $H'_N$.  We have denoted with $p'_N $  and $p''_N $  the orthogonal projection   over $H'_N$ and $H''_N$.
In this section we will  use the simpler notation  $u'$ and $u''$ instead of  $p'_N(u) $  and $p''_N(u) $.
 {Here and in the rest of  the paper we assume that $H$ is a Hilbert space on $\C$,  and   that the inner product in $H$ satisfies $\l au,\, bv\r =a\bar{b}\l u,\,v\r$ whenever $u,\, v\in H$ and  $a,b\in \C$}

  \medskip
  Theorem  \ref{T-2main} follows from Theorems \ref{T-riesz} and Theorem \ref{T-frame} below.

   \begin{Thm}\label{T-riesz}
 $\B_N$ is  a Riesz sequence of $H$  if and only if    the set  $W'_N=\{  w_1'  , ...,   w_N' \}$ is  linearly independent.
 \end{Thm}
 \begin{Thm}\label{T-frame}
  $\B_N$ is  a frame  of $H$  if and only if    $W'_N$ spans $H'_N$.

 \end{Thm}
Since $H'_N$ is a $N$ dimensional subspace of $H$,  the set $W'_N$ is  linearly independent if and only if it spans $H'_N$.
 Thus, Theorems     \ref{T-riesz} and \ref{T-frame} show that
 $\B_N$ is  a Riesz basis   if and only if  it is a frame or a Riesz sequence, and   Theorem \ref{T-2main} follows.

 \medskip

 \begin{proof}[Proof of Theorem \ref{T-riesz}]

Assume that  $\B_N$ is a Riesz sequence.    We prove that the $w'_j$'s are linearly independent, i.e,  that  $\sum_{j=1}^N c_jw'_j=0 $ if and only if the $c_j$'s are $=0$.

Every $w\in H$ can be written as   $w=w'+w''$, with $w''  =\sum_{k=N+1}^\infty  \l w' , \, v_k\r v_k$.
Thus,
$$
\Big\Vert \sum_{j=1 }^N c_j w_j'\Big\Vert^2
= \Big\Vert  \sum_{j=1 }^N c_j w_j -\sum_{j=1 }^N c_j w_j''\Big\Vert^2=
\Big\Vert  \sum_{j=1 }^N c_j w_j -\sum_{k=N+1}^\infty \!\!\! \big(\sum_{j=1}^N c_j  \l w''_j, \, v_k\r\big) v_k \Big\Vert^2\!.
$$
Since $\B_N$ is a Riesz sequence, we  have that
$$\Big\Vert \sum_{j=1 }^N c_j w_j'\Big\Vert^2\ge A_N\left(\sum_{j=1 }^N|c_j|^2+ \sum_{k=N+1}^\infty \big|\sum_{j=1}^N c_j  \l w''_j, \, v_k\r\big|^2\right)\ge  A_N \sum_{j=1 }^N|c_j|^2
$$
which shows that  the $w'_j$'s are linearly independent.

 \medskip
We prove  that  if  the $w'_j$'s are  linearly independent, then $\B_N$ is a Riesz basis of $H$ (and so it is  also a Riesz sequence).
To do this, we show that every $f\in H$ has a unique representation in terms of  elements of $\B_N$.

 The set   $\{w'_1, \, ...,\, w'_N\}$ is a basis of  $H_N'$ and the set $\{v_{N+1},\, v_{N+2},\, ...\} $ is an orthonormal basis of $H''_N$.    Thus, we can write $f=f'+f''$, with  $f'=\sum_{j=1}^N c_j  w'_j $ and $f''= \sum_{k=N+1}^\infty d_kv_k$.
Arguing as in the first part of the proof, we   obtain
 $$f=  \sum_{j=1}^N c_j  w'_j + \sum_{k=N+1}^\infty d_kv_k=
 \sum_{j=1}^N c_j  w _j + \sum_{k=N+1}^\infty \left( d_k- \sum_{j=1}^N c_j  \l w''_j, \, v_k\r\right)v_k
 $$
which gives a representation  of $f$ in terms of   elements of $\B_N$.  It is trivial to show that
this representation is unique and so that $\B_N$ is a Riesz basis.

 \end{proof}

 To simplify notation  in the proof of    Theorem \ref{T-frame}  and   Lemma \ref{L-prelim-reduction}  below,   we let
  $\hat H_N= \mbox{span}\{v_1, ...,\, v_N,\, w_1, ...w_N\}$;  we also let $\hat H_N'= \hat H_N\cap H' _N$
 and    $\hat H''_N= \hat H_N\cap H''_N$.   Clearly, $\hat H_N' =  H_N'$  and    $\hat H''_N=\mbox{span}\{ w''_1, ...,\, w''_N\}$.
We let    $S_H=\{ f\in H \ : \  ||f||=1\}$.

 We prove  the following easy, but important
\begin{Lemma}\label{L-prelim-reduction}
 $\B_N$ is a frame of $H$  if and only if   for every $f\in S_H\cap \hat H_N $, we have that
\begin{equation}\label{e-basicIneq}
 \sum_{j=1 }^N ( |\l f, w_j\r|^2- |\l f, v_j\r|^2 ) +1 >0.
\end{equation}
 The frame bounds of $\B_N$ are the maximum and minimum of  the functional $f\to \sum_{j=1 }^N ( |\l f, w_j\r|^2- |\l f, v_j\r|^2 ) +1$ in $S_H \cap \hat H_N$.
\end{Lemma}

 \begin{proof}  Let   $T_Nf=(\l f,\, w_1\r, \, ...,\, \l f,\, w_N\r, \, \l f,\, v_{N+1}\r, ...\,)$ be the analysis operator of $\B_N$.  By  Plancherel's theorem,   $1= \sum_{j=1 }^\infty  |\l f, v_j\r|^2$  for every  $f\in S_H$,  and
 $$
 ||T_Nf||_{\ell^2}^2=  \sum_{j=1 }^N  |\l f,\, w_j\r|^2+\sum_{k=N+1 }^\infty  |\l f,\, v_k \r|^2=
 \sum_{j=1 }^N ( |\l f,\, w_j\r|^2- |\l f,\, v_j\r|^2 ) +1.
 $$
Let    $\Pi_N :S_H  \to [0, \infty) $ be the functional $$f\to ||T_Nf||_{\ell^2}^2=  \sum_{j=1 }^N ( |\l f,\, w_j\r|^2- |\l f,\, v_j\r|^2 ) +1 .$$
 Thus, $\B_N$ is a frame with frame bounds $A_N $ and $B_N $ if and only if $A_N\leq \Pi_Nf \leq B_N$
 for every $f\in S_H$.

For $f\in S_H$,  we can write $f=g+h$, where $g $  is the orthogonal projection of $f$ on $\hat H_N  $
 and $h$ is the projection on the orthogonal complement of $\hat H_N$.
 Since $\Pi_N(f)=\Pi_N(g)$,
 the maximum and minimum of $\Pi_N$ on $S_H$ are attained on $S_H\cap \hat H_N$.

If  $\B_N$ is a frame, we have that $A_N\leq ||T_Nf||_{\ell^2}^2=\Pi_Nf\leq B_N$ whenever $f\in S_H\cap \hat H_N$, which proves \eqref{e-basicIneq}.

Assume now that $\Pi_N>0$ for every $f\in S_H \cap \hat H_N$.
Since $\Pi_N$ is continuous and positive in $H$ and  $S_H \cap \hat H_N$ is compact  (because it is closed and bounded in a finite dimensional   subspace),    by Weierstrass' theorem
there exist constants $0<A_N\leq B_N<\infty $ for which $A_N\leq \Pi_Nf\leq B_N$ whenever $f\in S_H \cap \hat H_N$, which proves that  $\B_N$ is   frame.
  \end{proof}

   \medskip

  \begin{proof}[Proof of Theorem \ref{T-frame}]

We prove both implications of Theorem \ref{T-frame} by contradiction.

If  $\B_N$ is a frame and  the $w'_j$ do  not span  $H_N'$, we can find  a function $\psi \in H'_N$ that is orthogonal to  $W'_N $. We can assume that $||\psi  ||=1$. Thus,
$1= \sum_{j=1 }^N |\l \psi ,\, v_j\r|^2$, and
\begin{eqnarray*}
 \Pi_N\psi & = &  \sum_{j=1 }^N  ( |\l \psi ,\, w_j\r|^2 -|\l \psi ,\, v_j\r|^2)+1\\
 & = & \sum_{j=1 }^N  ( |\l \psi ,\, (w_j'+w_j'')\r|^2 -|\l \psi ,\, v_j\r|^2)+1\\
 & = & \sum_{j=1 }^N    |\l \psi ,\, w_j'\r|^2
   = 0.
 \end{eqnarray*}
 which, by Lemma \ref{L-prelim-reduction} is a contradition.

 If the $w'_j$ spans $H_N'$   and   $\B_N$ is not a frame, by   Lemma \ref{L-prelim-reduction}
  we have that $\min_{f\in S_H\cap \hat H_N} \Pi_Nf=0$. Let $h \in S_H\cap \hat H_N$ be such that
  $
  \Pi_N h= 0$. We can write $h=h'+h''$, where  $h'$  is the orthogonal projection of $h$ on $  H_N'  $.  Since $1=||h||^2=||h'||^2+||h''||^2$ and $||h'||^2=\sum_{j=1 }^N |\l   h, v_j\r|^2 $, we have that
  $$0=\Pi_Nh= \sum_{j=1 }^N  ( |\l (h'+h'') ,\, w_j\r|^2 -|\l h' ,\, v_j\r|^2)+1
  $$$$=
  \sum_{j=1 }^N   |\l h'+h'',\, w_j\r|^2 +||h''||_2^2 . $$
Necessarily, all terms  in the sum equal 0 and so we have
  $h''=0$ and  $  |\l h' ,\, w_j \r|^2 $ $ = |\l h' ,\, w_j'\r|^2  =0$ for every $j=1$, ...,\,$N$.
Thus,  $h'$ is orthogonal to each $w'_j$ but since we have assumed that the $w'_j$ span $H'_N$, we have that $h'=0$, which is a contradiction. The theorem is proved.

  \end{proof}

 Recall that we have denoted with $\tilde B_N$ the set $  \{w'_1, ...,\, w'_N, v_{N+1},  v_{N+2}, ...\}$. From the proof of Theorem \ref{T-riesz} we can easily infer the following

  \begin{Cor}\label{C-reductions-basis}
  $\B_N$ is a Riesz basis  of  $H$ if and only if $\tilde B_N$  is a Riesz basis of $H$.
  \end{Cor}

  \begin{proof}  By Theorem \ref{T-2main}, both  $\B_N$  and $\tilde B_N$ are   Riesz bases   if and only if the set $W'_N=\{w'_1, \, ...,\, w'_N\}$ is linear independent.

  \end{proof}

 \noindent

\section{Proof of Theorem \ref{T-3main}}

We let $ \dsize  U_N$ be the  $N\times N$ matrix whose elements are $ u_{i,j}=\l  w _i,\, w _j\r$. Note that $ U_N= U_N'+ U_N''$ where $\dsize  U_N'=\{\l  w' _i,\, w' _j\r\}_{1\leq i,\, j\leq N}$ and $\dsize  U_N''=\{\l  w'' _i,\, w'' _j\r\}_{1\leq i,\, j\leq N}$.

We recall that by Theorem \ref{T-2main} and Corollary \ref{C-reductions-basis}, $\B_N$ and $\tilde \B_N$ are  Riesz bases of $H$ if and only if the set $W_N'=\{w'_1,\, ...,\, w'_N\}$ is a  basis of  $H_N'$.
\medskip
  Theorem \ref{T-3main} follows from Theorems \ref{T-const-tB} and \ref{T-const-B} below and Remark \ref{r-proofT3}.
\begin{Thm}\label{T-const-tB}
Assume that $W_N'$ is a basis of  $H_N'$. Then,
 $\tilde B_N$ is a Riesz basis of  $H$ with frame constants
\begin{equation}\label{e-const11}
  \tilde B_N=\max_{ |\vec c\,|\leq 1  }    \Big\{\l  U_N'\, \vec c,\     \vec c \,    \r +1- |\vec c\,|^2\Big\},\
 \tilde A_N=\min_{ |\vec c\,|\leq 1  } \Big\{   \l  U_N'\, \vec c,\     \vec c \,    \r +1- |\vec c\,|^2\Big\}.
 \end{equation}

 Let    $\lambda_N$ and $\Lambda_N$ be the maximum and minimum eigenvalue of the matrix  $\dsize  U_N'= \{\l  w'_s,\, w'_j\r\}_{1\leq j,\, s\leq N} $.   Then $\tilde B_N= \max\{\Lambda_N,\,1\} $ and $\tilde A_N=
\lambda_N $.
 \end{Thm}

 \begin{Thm}\label{T-const-B}
Assume that $W_N'$ is a basis of  $H_N'$. Then,
 $  B_N$ is a Riesz basis in $H$ with frame constants
\begin{equation}\label{e-const111}
   B_N=\max_{ |\vec c\,|\leq 1  }   \Big\{ \l  U_N'\, \vec c,\     \vec c \,    \r +\left(
   \sqrt{1- |\vec c\,|^2}+\sqrt{ \l  U_N''\, \vec c,\     \vec c\,\r }\right)^2\Big\},$$$$
A_N=\min_{ |\vec c\,|\leq 1  }   \Big\{ \l  U_N' \, \vec c,\     \vec c \,    \r +\left(
   \sqrt{1- |\vec c\,|^2}-\sqrt{ \l  U_N''\, \vec c,\     \vec c\,\r }\right)^2\Big\}.
 \end{equation}   \end{Thm}

  \medskip

 \begin{proof}[Proof of Theorem \ref{T-const-tB}]
Recall that the frame constants of $\tilde\B _N$ are the supremum and    infimum  of the function \begin{equation}\label{e-defF} F(\vec b, \vec c)=
\Big\Vert \sum_{j=1}^N c_jw'_j+ \sum_{k=N+1} ^\infty  b_kv_k \Big\Vert^2,
\end{equation}
 where $\vec c=  (c_1, ...,\,c_N)$ and $\vec b=( b_{N+1},  \,  b_{N+2},...)$ are such that $|\vec c\, |^2+ |\vec b\,|^2=1$.
  Since
  $$
  F (\vec b, \vec c)=  \Big\Vert \sum_{j=1}^N c_jw'_j \Big\Vert^2+
   \Big\Vert  \sum_{k=N+1}^\infty  b_kv_k \Big\Vert^2,
   $$
   $$=  \sum_{i,j=1}^N c_i\bar c_j\l w'_i,\, w'_j\r  +\sum_{k=N+1}^\infty |b_k|^2
   = \l  U_N'\, \vec c,\     \vec c \,    \r +1-|\vec c\,|^2
   $$ we  have \eqref{e-const11}.  Note that  $F(\vec b, 0)=1 $ whenever $|\vec b|=1$, and so  $\tilde B_N\ge  1$  and $\tilde A_N\leq  1$.
  By a well-known theorem of linear algebra,  the maximum and   minimum  eigenvalue of $ U_N'$ are maximum and minimum of the  Hermitian quadratic form $\vec c\to \l  U_N'\,\vec c,\     \vec c \,    \r$ when $|\vec c\,|=1  $.
 Thus,
 $\lambda_N|\vec c\,|^2 \leq \l  U_N'\, \vec c,\     \vec c \,    \r \leq \Lambda_N |\vec c\,|^2$ for every $\vec c\in\C^N$.  Using this inequality and the previous observation,  from   \eqref{e-const11} we  obtain:
  $\tilde B_N= \max\{\Lambda_N, 1\}$ and $\tilde A_N=
 \min \{\lambda_N, 1\}$. Finally, we can chose  $\vec c$ and $\vec b$  so that $F(\vec b,\, \vec c)=||w'_j||^2$; thus,
\begin{equation}
\tilde A_N \leq \min_{1\leq j\leq N}||w'_j||^2\leq 1,\label{tAN}
\end{equation}
 from which follows that  $\tilde A_N =\lambda_N$.
   \end{proof}

\begin{proof}[Proof of Theorem \ref{T-const-B}]
 The frame constants of $\B _N$ are the supremum and    infimum  of the function
  $$
  G(\vec b, \, \vec c\,)= \Big\Vert \sum_{j=1}^N c_jw_j+ \sum_{k=N+1}^\infty  b_kv_k \Big\Vert^2
  $$
  where $|\vec c\, |^2+ |\vec b\,|^2=1$.  We gather:
 $$
 G(\vec b, \, \vec c\,) =\Big\Vert \sum_{j=1}^N c_j \omega_j  \Big\Vert^2 +   \sum_{k=N+1} ^\infty |b_k| ^2 +
   2\Re   \sum_{j=1}^N  \sum_{k=N+1}^ \infty  c_j \bar b_k\l w  _j,\, v_k\r
  $$
   \begin{equation}\label{e-int1} =    \l  U_N\vec c,\    \vec c \,  \,  \,  \r +1-|   \vec c \,  \,  \,  |^2 +2\Re   \sum_{j=1}^N  \sum_{k=N+1}^ \infty  c_j \bar b_k\l w'' _j,\, v_k\r.
   \end{equation}
   $$= \l  U_N\vec c,\    \vec c \,  \,  \,  \r +1-|   \vec c \,  |^2 +2\Re   \l \vec b,\, \vec \rho(\vec c )\r
   $$
   where we have  let
  $\dsize \vec \rho(\vec c)=\Big(\sum_{j=1}^N c_j \l w'' _j,\, v_{N+1}\r,\,   \sum_{j=1}^N c_j \l w'' _j,\, v_{N+2}\r,\, ... \Big)$.

Fix  $\vec c \in \C^N $,  with $ |\vec c\,|=1 $;   the function   $\vec b \to
    2\Re   \l \vec b,\, \vec \rho(\vec c )\r
   $
 is  maximized and minimized  when $\vec b$ is a constant multiple of  $\vec \rho(\vec c)$. Since $|\vec b\,|^2=1-|\vec c\,|^2$, we  can choose $\vec b= \sqrt{1-|\vec c\,|^2} \frac{\vec \rho(\vec c)}{|\vec \rho(\vec c)|}$. Thus, for every   $\vec c\in \C^N $   and for every $\vec b\in \ell^2$ with $|\vec b\,|^2= 1 -|\vec c\,|^2$,
 we have that
\begin{equation}\label{e-in1}
   \l  U_N\vec c,\    \vec c \,    \r +1-|   \vec c \,  |^2-2 |\vec \rho(\vec c)|\,\sqrt{1-|\vec c\,|^2}\leq G(\vec c,\, \vec b\,).
  \end{equation}
    and
   \begin{equation}\label{e-in2}
    G(\vec c,\, \vec b\,)\leq \l  U_N\vec c,\    \vec c \,    \r +1-|   \vec c \,  |^2+2 |\vec \rho(\vec c)|\,\sqrt{1-|\vec c\,|^2}.
  \end{equation}
  Let us evaluate the  norm of $\vec\rho(\vec c)$. If we denote with
   $A$ the   matrix whose $(k-N)$-th  row is   $  (\l w''_1,\, v_k\r,\, ...,\, \l w''_N,\, v_k\r)$,
  we can write
   $$
  | \vec\rho(\vec c)|^2=  \sum_{k=N+1}^\infty \Big|\sum_{j=1}^N    c_j  \l w'' _j,\, v_k\r\Big| ^2=
   || A\vec c\,||^2=\l A^*A \vec c,\ \vec c\,\r.
    $$
    We have  $A^*A=\{\alpha_{j,s}\}_{1\leq j,s\leq N}$, with $\dsize \alpha_{j,s}= \sum_{k=N+1}^ \infty \l w_j'',\, v_k\r \overline{\l w_s'',\, v_k\r }.
   $
   Since $\{v_k\}_{k \geq N+1}$ is an orthonormal basis on $H_N''$, we can easily verify (or see also \eqref{e-Planch}) that
   $\alpha_{j,s}=\l w''_j, \, w''_s\r$.
   Thus, $A^*A=  U_N''$ and
   $
   | \vec\rho(\vec c)|^2 =   \l  U_N''\vec c,\    \vec c \,    \r .
   $

 From \eqref{e-in1} and \eqref{e-in2} follow  that
    $$
    \l  U_N\vec c,\   \vec c \,  \,  \r +1-|\vec c\, |^2 - 2 \sqrt{(1-|\vec c\,|^2) \l  U_N''\vec c,\ \vec c\r } \leq  G(\vec c,\, \vec b\,)
    $$
    and
    $$
      G(\vec c,\, \vec b\,)\leq \l  U_N\vec c,\   \vec c \,  \,  \r +1-|\vec c\, |^2 + 2 \sqrt{(1-|\vec c\,|^2) \l  U_N''\vec c,\ \vec c\r }.
    $$
    Since $ U_N= U_N'+ U_N''$, \eqref{e-const111}  easily follows.

  \end{proof}

  \medskip
   \noindent
   {\it  Example}.  In the simple case $N=1$, we have that $ U_1'=\{||w'_1||^2\}$ and $U_1''=\{||w''_1||^2\}$; by Theorem \ref{T-const-tB},
   $\tilde A_1=||w_1'||^2$ and $ \tilde B_1=1$.

By \eqref{e-const111},  $A_1$   is the   minimum of the function   $$f(c)=   ||w_1'||^2 c^2+\left(
   \sqrt{1-  c ^2}- c||w_1''||\right)^2 $$$$=  ||w_1'||^2 c^2+1-c^2 +||w_1''||^2 c^2 -2 ||w_1''||^2  c\sqrt{1-c^2}.
   $$
   Since $ ||w_1'||^2   +||w_1''||^2=1$, we have that
   $
   f(c)=1-2 ||w_1''||^2  c\sqrt{1-c^2}
   $  which is minimized when $c^2=\frac{1}{  2}$. Thus,
   $A_1= 1-||w''_1||^2$. The calculation of $B_1$ is similar; we obtain  $B_1= 1+||w''_1||^2$.
\medskip

 We prove the following

 \begin{Cor}\label{C-new1}
 The  frame constants   of $\tilde \B_N$ are the  minimum and maximum singular values of the matrix $M_N=\{\l w_j,\, v_k\r\}_{1\leq j,k\leq N}$, or:
  $$
  \tilde A_N=\min_{|\vec c|=1}  ||M_N\vec c\,||; \quad \tilde B_N=\max_{|\vec c|=1}  ||M_N\vec c\,||.
   $$
      Furthermore,  $\tilde B_N\leq \tilde B_{N+1}$  for every $N\ge 0$.
   \end{Cor}

   \begin{proof}
    We can write $w'_j= \sum_{k=1}^N p_{j,k} v_k$, with $\dsize p_{j,k} =\l \omega_j, v_k\r $.
 It is easy to verify that $\dsize U'_N=\{\l  w'_j, \,w'_k\r\}_{1\leq j,k\leq N}=M_N M_N^*$.
   Indeed,
   $$
   \l  w'_j, \,w'_i\r=\l \sum_{k=1}^N p_{j,k}v_k,\ \sum_{k=1}^N p_{i,k}v_k\r= \sum_{k=1}^N p_{j,k}\overline{p_{i,k}}.
   $$
   The   first part of the corollary follows from  Theorem \ref{T-3main} and standard linear algebra results.

    To prove the second part, we observe that $M_N$ is a sub-matrix of $M_{N+1}$. It is well known (see e.g. \cite{Thompson} )
that the eigenvalues and singular values of a non-negative definite matrix interlace with those of a sub-matrix.  If we denote with $\sigma_1(A)$ the largest singular value of a  matrix $A$ and  if we denote with $B$ a sub-matrix of $A$, we have that $\sigma_1(B)\leq \sigma_1(A)$. With the notation of Theorem \ref{T-3main}, we have that  $\tilde B_N=\max\{\Lambda_N, \,1\}\leq \tilde B_{N+1}=\max\{\Lambda_{N+1}, \,1\}$, as required.

   \end{proof}

   \section{Proof of Theorem \ref{T-first-main}}

\begin{proof}[Proof of Theorem \ref{T-first-main}]

 Let $\dsize A=   \liminf_{N\to\infty }A_N  $ and $ B=\dsize \limsup_{N\to\infty }B_N  $. We prove  that    if $  A > 0$ and $   B< \infty  $ then $\B_\infty$  is a  Riesz basis.

We prove first that $\B_\infty $ is a Riesz sequence:
for every finite family $a_1$, ... $a_N \in \C$, with $\sum _{j=1}^N |a_j|^2=1$,  we have
$$
  A_N\leq  || a_1 w_1+...+a_N w_N||^2\leq B_N.
$$
So, for every $N\ge 1$, we also have \begin{equation}\label{e-ineq1}
  A \leq   A_N\leq    || a_1 w_1+...+a_N w_N||^2\leq B_N\leq  B
\end{equation}
as required.

\medskip

Let us prove that  $\B_\infty$ is a frame. By     \cite[Thm. 7.4]{Heil} the right inequality
 in \eqref{e-ineq1}  implies the right-frame inequality:
 \begin{equation}
  \sum_{n=1}^\infty |\l f, w_n\r|^2\leq B  ||f||^2,\label{Bf}
\end{equation}
  for every $f\in H$. We prove that  the left frame inequality holds with constant $A $.

Let   $
 T_N f=(\l f, w_1\r, \, ..., \l f, w_N\r,\, \l f, v_{N+1}\r, \l f, v_{N+2}\r, \,...)
 $ be the   analysis operator  of $\B_N$,
 and let $T_\infty$ be the analysis operator of $\B_\infty$.
We prove that     $T_\infty $  is the pointwise limit of the $T_N$'s, i.e., that $
 \lim_{N\to\infty} ||  T_\infty   f-T_Nf||_{\ell^2}=0$ for every $f\in H$. Indeed,
 $$
 \lim_{N\to\infty} ||  T_\infty   f-T_Nf||_{\ell^2}= \lim_{N\to\infty}\left(\sum_{k=N+1}^\infty  |\l  w_k-v_k,\, f\r|^2\right)^{\frac 12} $$$$\leq \lim_{N\to\infty} \left(\sum_{k=N+1}^\infty  |\l  w_k, \, f\r|^2\right)^{\frac 12}+  \lim_{N\to\infty} \left(\sum_{k=N+1}^\infty  |\l  v_k, \, f\r|^2\right)^{\frac 12}=0
 $$
 because, by \eqref{Bf},  the sums above   are   "tails" of convergent series.
 Thus, for every $f\in H$,
 $$||T_\infty   f||_{\ell^2}^2   =   \lim_{N\to\infty}|| T_Nf ||_{\ell^2}^2.
 $$

Fix $f\in H$, with $||f||=1$.   Since
$$\liminf_{N\to\infty} ||T_N f||_{\ell^2}^2\leq \lim_{N\to\infty} ||T_N f||_{\ell^2}^2 \leq  \limsup_{N\to\infty} ||T_N f||_{\ell^2}^2  $$
and  $A_N\leq ||T_N f||_{\ell^2}^2\leq B_N$, we have that
   $$\liminf_{N\to\infty} A_N \leq \lim_{N\to\infty} ||T_N f||_{\ell^2}^2=||T_\infty f||_{\ell^2}^2  \leq  \limsup_{N\to\infty} B_N  $$
as required.

  \medskip
  We prove the second part of   Theorem\ref{T-first-main}. Assume that $\B_\infty$ is a Riesz basis with    frame constants $A_{\infty},B_{\infty}$.
Recall that  $A_\infty=\inf_{||f||=1} ||T_\infty f||_{\ell^2}^2 $ and $B_\infty=\sup_{||f||=1} ||T_\infty f||_{\ell^2}^2 $.

We  fix $\epsilon>0$ and $ \bar f =\bar f(\epsilon)$ satisfying $||\bar f||=1$,       such that  $ ||T_\infty \bar f||_{\ell^2}^2\ge  B_\infty-\epsilon .$
Since the series  $\sum_{k=1}^\infty  |\l  w_k,\, \, f\r|^2 $ and $\sum_{k=1}^\infty  |\l  v_k, \, f\r|^2 $ converge for every $f\in H$,
we can chose  $\bar N>0$ sufficiently large so that
 $\sum_{j=N+1}^\infty|\l \bar f,\ v_j\r|^2<\epsilon $ and $\sum_{j=N+1}^\infty| \l \bar f,\ w_j\r|^2<\epsilon$ whenever $N>\bar N$. Thus,
$$
B_\infty-\epsilon \leq ||T_\infty \bar f||_{\ell^2} ^2 \leq  \sum_{j=1 }^N |\l \bar f,w_j\r|^2 +\epsilon  \leq   || T_N\bar f||_{\ell^2} ^2 +\epsilon \leq  B_N + \epsilon.
$$
We have proved that  for every $\epsilon>0$ exists $\bar N>0$ such that
$B_\infty\leq  B_N+2\epsilon$ for every $N>\bar N$; thus,
$$ B_\infty\leq \sup _{N>\bar N} B_N+2\epsilon, $$ and
$$B_\infty \leq \lim _{\bar N\to\infty}   \sup _{N>\bar N} B_N=\limsup_{  N\to\infty} B_N.$$
The proof of the other inequality is very similar:  we can find
$ \bar f $  of  unit norm  such that $|| T_\infty\bar f||_{\ell^2} ^2 \leq  A_\infty+\epsilon$;    we can find $\bar N>0$ such that, for every $N>\bar N$,
$$
A_\infty+\epsilon \ge || T_\infty\bar f||_{\ell^2} ^2 \ge \sum_{j=1 }^N |\l \bar f,w_j\r|^2    \ge || T_N\bar f ||_{\ell^2} ^2- \epsilon \ge  A_N - \epsilon.
$$
We obtain  $A_\infty> A_N-2\epsilon$  for every $N>\bar N$; we can argue as in the previous proof to conclude that  $A_\infty \ge \liminf_{N\to\infty} A_N$.

 \end{proof}

  \begin{Rem}\label{Rem-J-example}
  The following example shows that the $ A_N$  need not form a monotonic sequence.
   Let $\{e_j\}_{j\in\N} $ be an orthonormal basis of $H$.  Let $\tau(0)=0$ and $\tau(n)=\sum_{j=1}^n j=\frac n2(n+1)$.
   We define the sequence $\{w_j\}_{j\in\N}$ as follows:   we let
   $$ w_{\tau(n-1)+1}= \frac 1{\sqrt{n}}\sum_{k=1}^n v_{\tau(n-1)+k}, \quad  n\ge 1.$$ When   $j\ne \tau(n-1)+1$ we chose the $w_j$  in such a way that the set \newline$\{w_{\tau(n-1)+1}, \ w_{\tau(n-1)+2},\, ...,\, w_{\tau(n-1)+n}\}$ forms an orthonormal basis of \newline $\mbox{span}\{e_{\tau(n-1)+1}, \ e_{\tau(n-1)+2},\, ...,\, e_{\tau(n-1)+n}\}$.  Thus, $w_j=w'_j$ and
 the $\{w_j'\}_{j\in\N}$ form an orthonormal basis of $H$. We have $ A_\infty= 1$, and
   $$ A_{\tau(n-1)+1}\leq |\l w_{\tau(n-1)+1} , \ e_{\tau(n-1)+1}\r|^2=\frac 1n$$
  which shows that the sequence $ A_n $ is not monotonic. Note that  $\dsize\liminf_{n\to\infty}  A_n=0$.

   \end{Rem}

\subsection{Limit of  frame constants }

Theorem \ref{T-first-main} shows that
$\B_\infty$ is a Riesz basis of $H$ whenever  $\liminf_{N\to\infty }A_N>0$ and  $ \limsup_{N\to\infty }B_N <\infty$.  The example in Remark \eqref{Rem-J-example} shows that we can have $A_\infty>0$ but $\liminf_{N\to\infty }A_N=0$.  If would be desirable to have  $ A_\infty= \liminf_{N\to\infty }A_N$ and $   B_\infty= \limsup_{N\to\infty }B_N$, but for that we need more stringent conditions on the $w_j$.


  We denote with   $\dsize |||L||| =  \sup_{||f||=1} ||L f ||_{\ell^2} $ the operator norm of  a linear operator $L:H\to \ell^2$.

The proof  of Theorem \ref{T-first-main} shows that if $\B_\infty$ is a Bessel sequence, i.e., if
$\sum_{j=1 }^\infty |\l \bar f,w_j\r|^2  <\infty$ for every $f\in H$, then
  $T_N f\to T_\infty  f$  for every $f\in  H$.
  Assume that  $T_N\to T_\infty$ uniformly, i.e. that
 for every  $\epsilon >0$ there exists $\bar N>0$ that depends only on $\epsilon $,   for which    $|| T_Nf-T_\infty f  ||<\epsilon$   whenever $||f||=1$; from
   $
||T_N f|| - \epsilon <  ||T_\infty  f||  < ||T_N f|| + \epsilon,
  $
we can see at once that  $$
 \sup_{||f||=1} ||T_N f|| - \epsilon <  \sup_{||f||=1} ||T_\infty  f||  < \sup_{||f||=1}||T_N f|| + \epsilon;
  $$
  $$
 \inf_{||f||=1} ||T_N f|| - \epsilon <  \inf_{||f||=1} ||T_\infty  f||  < \inf_{||f||=1}||T_N f|| + \epsilon,
  $$
  and so   $\lim_{N\to\infty} B_N=B_\infty$  and  $\lim_{N\to\infty} A_N=A_\infty$. This observation proves the following

    \begin{Thm}\label{L-limit-TN} If $\B_\infty$ is  a Riesz basis   and if
  $ T_N\to T_\infty  $ uniformly, then   $\dsize\lim_{N \to \infty} B_N=B_\infty$ and $\dsize \lim_{N\to \infty} A_N=A_\infty$.
 \end{Thm}

    A necessary condition for  the uniform convergence of the $T_N $'s is given by the following
\begin{Lemma}\label{L:-Nec-unif-convergence}

a) If $ T_N\to T_\infty  $ uniformly, then $\dsize\lim_{N\to\infty}  ||v_N-w_N||= 0$.

b) If $\{ ||v_n-w_n||\}_{n\in\N} \in \ell^2$, then $ T_N\to T_\infty  $ uniformly.
 \end{Lemma}

  \begin{proof}
Let $\epsilon >0$ and let $\ov N>0$ such that    $ ||| T_N-T_\infty|||<\epsilon$  for every $N>\bar N$;  then, for every $||f||=1$,
\begin{equation}
||(T_\infty-T_N)f||^2=\sum_{N+1}^{\infty}|\l f,v_j-w_j\r|^2 <\epsilon^2. \label{TNT}
\end{equation}
 Hence, for all $m>N> \bar N$,
$$||v_m-w_m||=\sup_{||f||=1} |\l f,\,v_m-w_m\r| < \epsilon.$$
This proves part a). Part b) follows easily from \eqref{TNT}.

 \end{proof}

 \medskip

   \section{Applications}
       Let $D$ be a bounded and measurable set of $\R^d$, { with $|D|<\infty$. We denote with   $\chi_D$   the characteristic function of  $D$.}
   In this section we find sufficient conditions that ensure that a set of exponentials $\B=\{e^{2\pi i \lambda_n\cdot x}\}_{n\in\N}$ is a Riesz basis of $L^2(D)$.  We denote  $e^{2\pi i \lambda_n\cdot x}$  by $e(\lambda_n)$
  for simplicity.
Let  ${\cal V}=\{v_j(x)\}_{j\in\N}$ be an orthonormal basis of $L^2(D)$.

We denote with $\hat f(y)=\int_{\R^n} f(x) e^{-2\pi i x\cdot y} dx $  the Fourier transform of $f\in L^2(\R^n)$.
We prove the following
\begin{Thm}\label{T-suff-basis}
Assume that there exist constants $a'>0,\ 0<a<1$,  and  $0<\delta<1$  such that, for  every $n\in\N$ and every $N\ge 1$, we have that
\begin{equation}\label{e-a} a\leq  |\widehat {\chi_D v_h}(\lambda_h)|\leq a';\quad  \sum_{j=1\atop{j\ne k}}^N|\widehat {\chi_D v_k}(\lambda_j)| + \sum_{j=1\atop{j\ne k}}^N |\widehat {\chi_D v_j}(\lambda_k)|
  \leq  2a(1-\delta)
 \end{equation}

 Then $ \B=\{e^{2\pi i \lambda_n\cdot x}\}_{n\in\N}$  is a Riesz basis of $L^2(D)$
   with frame constants
  $A_\infty \ge a\delta$, $B_\infty \leq a'+a(1-\delta)$.
\end{Thm}

 \begin{proof}
As in the first sections of this paper, we let $\B_0={\cal V} $, {$\B_\infty=\{e(\lambda_n)\}_{n\in\N}$} and   $\B_N
=\{e(\lambda_1),...,\, e(\lambda_N),\, v_{N+1}(x),\, ...\}$. We also let $\tilde \B_N
=\{p_N(e(\lambda_1)),...,\, p_N(e(\lambda_N)),\, v_{N+1}(x),\, ...\}$ where $p_N$ is  the projection over $H_N= span\{v_1, ...,\, v_N\}$.
By Theorem \ref{T-2main}, $\B_N$ is a Riesz basis of $L^2(D)$   if and only if $\tilde \B_N$ is a Riesz basis. By Theorem \ref{T-first-main} and Corollary \ref{C-new1}, if all the $\tilde \B_N$ are Riesz bases with frame constants    $\tilde A_N$ and $\tilde  B_N$, and if $\liminf_{N\to\infty} \tilde A_N>0$ and $\lim_{N\to\infty} \tilde B_N<\infty$, then also $\B_\infty=\B$ is a Riesz basis with frame constants $A_\infty \ge \liminf_{N\to\infty} \tilde A_N$, $B_\infty \leq \lim_{N\to\infty} B_N$.
Corollary \ref{C-new1} shows that  the frame constants of the  $\tilde \B_N$ are the maximum and minimum  singular values of  the matrix $M_N$ whose elements are
$$
m_{j,k}=\l e(\lambda_j), v_k\r= \widehat{\chi_D v_k}(\lambda_j), \quad 1\leq j,\, k\leq N.
$$
 By Theorem   \ref{T-Qi} in Appendix 2,
 $$A_N\ge a-a(1-\delta)=a\delta, \quad B_N\leq a'+ a(1-\delta)
 $$ for every $N\ge 1$, and so the same estimate hold also for $\tilde A_\infty$ and $\tilde B_\infty$. The  proof of Theorem \ref{T-suff-basis} is concluded.
   \end{proof}

  \begin{Rem} The constants $a$  and $\delta$ in Theorem \ref{T-suff-basis} cannot be completely arbitrary. When  $|D|=1$,
   by Plancherel theorem  we have that
  $$
  1=||e(\lambda_k)||_{L^2(D)}^2 =\sum_{j=1}^\infty|\l v_j, \, e(\lambda_k)\r_{L^2(D)}|^2
  =
  |\widehat {\chi_D v_k}(\lambda_k)|^2+ \sum_{j=1\atop{j\ne k}}^\infty |\widehat {\chi_D v_j}(\lambda_k)|^2.$$
  Since each term in the sums above  is   $\leq 1$,  we have that
  $$
  \sum_{j=1\atop{j\ne k}}^\infty |\widehat {\chi_D v_j}(\lambda_k)|^2\leq   \sum_{j=1\atop{j\ne k}}^\infty |\widehat {\chi_D v_j}(\lambda_k)| \leq
   2a(1-\delta).
  $$
  Thus,
  $ \dsize
  |\widehat {\chi_D v_k}(\lambda_k)|^2 \ge 1-2a(1-\delta).
  $
  If $1-2a(1-\delta) \ge a^2 , $ i.e. if $$0< a\leq \sqrt{1+(1-\delta)^2  }-(1-\delta), $$we    have
  $ |\widehat {\chi_D v_k}(\lambda_k)| \ge a $.

  \end{Rem}
  \section{Appendix}

\subsection{Riesz sequences, frames and bases}
 We have used the excellent textbooks \cite{Heil}  and \cite{Cr} for most of the definitions  and preliminary results presented in this section.

 Let $H$ be a  separable Hilbert space  with inner product $\langle\ ,\ \rangle $  and norm $||\ ||=\sqrt{\l \ , \    \r} $.
A sequence of vectors ${\mathcal V}= \{v_j\}_{j\in\Z} \subset H $   is a
 {\it frame} if
there exist  constants $0< A, \ B<\infty$    such that the following inequality holds  for every $w\in H$.
\begin{equation}\label{e2-frame}
 A||w||^2\leq  \sum_{j\in\Z} |\l  w, v_j\r |^2\leq B ||w||^2.
\end{equation}

  We say that  ${\cal V}$  is a   {\it  tight frame } if $A=B$   and   is a {\it Parseval frame} if $A=B=1$.


The left inequality in \eqref{e2-frame} implies that  ${\cal V}$ is     complete    in $H$  but it may not be linearly independent: we say that
${\mathcal V}$ is a  {\it Riesz sequence}  if   there exists constants $0<A\leq B <\infty$ such that, for every finite set of coefficients $ \{a_j\}_{j\in J}\subset\C $,
we have that
\begin{equation}\label{e2- Riesz-sequence}
 A  \sum_{j\in J}   |a_j|^2   \leq  \left\Vert \sum_{j\in J}  a_j  v_j \right\Vert^2  \leq B \sum_{j\in J} |a_j|^2,
\end{equation}

A {\it Riesz basis} is a  frame and a Riesz sequence, i.e. it is a set ${\cal V}$ that satisfies \eqref{e2- Riesz-sequence}  and \eqref{e2-frame}. It is proved in \cite[Prop. 3.5.5]{Cr}  that the  constants $A$ and $B$ in \eqref{e2-frame} and \eqref{e2- Riesz-sequence}  are the same.

Equivalently, a  Riesz basis is a bounded unconditional Schauder basis. Thus, every element in $H$ has a unique representation as a linear combination of elements of  ${\cal V}$.

 An orthonormal  basis   is a Riesz basis; we can write   $w=\sum_{j\in\Z} \l v_j,  \, w\r v_j$ for every $v\in H$ and   this  representation formula  yields  the following important identities: for every $    w,\ z\in  H$,
\begin{equation}\label{e-Planch}
||w||^2= \sum_{n\in\Z } |\l v_n,\, w\r|^2, \quad \l w, z\r=   \sum_{n\in\Z }  \l v_n,\, w\r\overline{\l z,v_n\r}.
\end{equation}
 We recall the following useful
 \begin{Prop}\label{prop-C}   \cite[Prop. 3.2.8]{Cr}  A sequence of unit vectors in $H$  is a Parseval frame if
and only if it is an orthonormal Riesz basis.
\end{Prop}
%
%

 \subsection{Gershgorin  theorem for  singular values }
Consider the complex matrix  $A=\{a_{i,j}\}_{1\leq i,j\leq n}$;
Let $\dsize R_i= \sum_{j=1\atop{j\ne i}}^n |a_{i,j}|$ and   $\dsize C_i=   \sum_{j=1\atop{j\ne i}}^n |a_{j,i}|$, and  $s_j=\max\{C_j, R_j\}$.
%

 Gershgorin's theorem   provides a powerful tool for estimating the eigenvalues of  complex-valued matrices.
%
It states that each eigenvalue of a square matrix  ${  M}=\{m_{i,j}\}_{1\leq i,j\leq n}$ is in at least one of the disks
$ D_j=\{z\in\C \, :\, |z-m_{j,j}|\leq R_j\},
$
and in at least one of the disks
$ D'_j=\{z\in\C \, :\, |z-m_{j,j}|\leq C_j\}.
$
See \cite{G}, and also   \cite[pg. 146]{MM} and \cite{BM}.

Observe that if    $|m_{j,j}|> s_j $ for every $j$, (i.e., if $M$ is {\it diagonally dominant}),   then $M$ is nonsingular.
  Gershgorin-type theorems for singular values are proved in \cite{Johnson, Qi}.
The   theorem below is stated for $n\times n$ matrix but it is also valid, with a slightly different statement, also for general matrices.

   \begin{Thm}\label{T-Qi}

 a) Each singular value of $A$ lies in one of the real intervals
 $$B_i= [(|a_{i,i}|-s_i)_+, \ |a_{i,i}|+s_i ], \quad 1\leq i\leq n $$
 where as usual $x_+=\max\{x, 0\}$.

 b) Let $\sigma_n(A)$ be the minimum eigenvalue of $A$. Then
 $$\sigma_n(A)\ge \min_{1\leq k\leq n} \{ |a_{k,k}|-\frac 12(R_k+C_k)\}.$$

 \end{Thm}
   For Part a) see Theorem 2 in \cite{Qi}, for  part b), see  Theorem 3 in \cite{Johnson}.


\begin{thebibliography}{999}






 \bibitem{AB} Anderson, N.; Best,  G.
{\it A Gerschgorin-Rayleigh inequality for the eigenvalues of Hermitian matrices.}
Linear and Multilinear Algebra 6 (1978/79), no. 3, 219--222.
%
\bibitem{BM}
Brualdi, R.; Mellendorf, S. {\it Regions in the complex plane containing the eigenvalues of a matrix}. Amer. Math. Monthly 101 (1994), no. 10, 975--985.

\bibitem  {Cr}   Christiansen, O.
{\it An Introduction to Frames and Riesz Bases }(second edition) Birkh\"auser, 2016.

\bibitem{Dc} De Carli, L.    {\it Concerning exponential bases on multi-rectangles of $\R^d$}
  (to appear in  the proceeding of the International Conference in Approximation Theory, Savannah, May 2017)


 \bibitem{DMT}    De Carli, L;   Mizrahi, A. and Tepper, A., {\it Three problems on exponential bases},   Canadian Math. Bullettin, DOI 10.4153/CMB-2018-015-6 (2018)


\bibitem{F}  Fuglede, B. {\it Commuting self-adjoint partial differential operators and a group theoretic
problem}, J. Funct. Anal. 16 (1974), 101--121.


\bibitem {G} Gerschgorin, S.  {\it \"Uber die Abgrenzung der Eigenwerte einer Matrix.} Izv. Akad. Nauk. USSR Otd. Fiz.-Mat. Nauk 6, (1931), 749--754.
  %
   \bibitem{GL}  Grepstad, S.;  Lev, N. {\it Multi-tiling and Riesz bases}. Adv. Math. 252 (2014), 1--6.
\bibitem{Heil}  Heil, C. {\it  A basis theory primer}, Appl.  Num. Harm. Analysis,
 Birkh\"auser  (2011).


\bibitem{K}   Kolountzakis, M. {\it The study of translational tiling with Fourier Analysis.}  {\it Fourier Analysis and
Convexity},  131--187. Birkh\"auser, 2004.

 \bibitem{K2}  Kolountzakis, M., {\it Multiple lattice tiles and Riesz bases of exponentials,}
Proc. Amer. Math. Soc. 143 (2015), 741--747.


\bibitem{Johnson} Johnson, C.R. {\it A Gersgorin-type lower bound for the smallest singular value}. Linear Algebra Appl. 112 (1989), 1–-7.


  \bibitem{MM}   Marcus, M.;  Mint, H.{\it A Survey of Matrix Theory and Matrix Inequalities},



 \bibitem{Qi}
Qi, L. Q., {\it
Some simple estimates for singular values of a matrix.}
Linear Algebra Appl. 56 (1984), 105–-119.

\bibitem{Thompson} Thompson, R. C., {\it Principal Submatrices IX: Interlacing Inequalities for Singular
Values of Submatrices } Linear Algebra Appl. 5  (1972), 1--12.

 \bibitem {Y} Young, R. M.  {\it An introduction to nonharmonic Fourier series}. Revised first edition. Academic Press, Inc., San Diego, CA, 2001.



\end{thebibliography}
\end{document}